\documentclass[a4paper,10pt,reqno, english]{amsart}

\usepackage{amsmath,amssymb,amscd,amsthm,amsfonts}
\usepackage{graphicx,subfigure}
\usepackage[colorlinks=true, allcolors=blue]{hyperref}
\usepackage{dsfont}
\usepackage{comment}
\usepackage[dvipsnames]{xcolor}
\usepackage{enumerate}
\usepackage[nobysame, alphabetic]{amsrefs}
\usepackage[capitalize,noabbrev]{cleveref}

\newtheorem{theorem}{Theorem}[section]

\newtheorem{lemma}[theorem]{Lemma}
\newtheorem{claim}[theorem]{Claim}
\newtheorem{corollary}[theorem]{Corollary}

\newtheorem{conjecture}[theorem]{Conjecture}

\newtheorem*{problem}{Problem}
\newtheorem{definition}{Definition}

\newcommand{\ff}{\mathcal{F}}
\newcommand{\suchthat}{\;\ifnum\currentgrouptype=16 \middle\fi|\;}

\def\rr{\mathds{R}}

\DeclareMathOperator{\tr}{tr}
\DeclareMathOperator{\conv}{conv}
\DeclareMathOperator{\vol}{vol}

\DeclareMathOperator{\dist}{dist}
\DeclareMathOperator{\boxx}{box}

\usepackage{tikz}
\usetikzlibrary{decorations.markings,arrows,automata,arrows,backgrounds}
\usetikzlibrary{
  knots,
  hobby,
  decorations.pathreplacing,
  shapes.geometric,
  calc
}

\tikzset{cross/.style={cross out, draw=black, minimum size=2*(#1-\pgflinewidth), inner sep=0pt, outer sep=0pt},
cross/.default={1pt}}

\title{Colorful and Quantitative Variations of Krasnosselsky's Theorem}
\author{Connor Donovan}
\address{Ursinus College. Collegeville, PA.}
\email{codonovan@ursinus.edu}

\author{Danielle Paulson}

\address{Harvard University. Cambridge, MA.}
\email{dpaulson@college.harvard.edu}

\author{Pablo Soberón}

\address{Baruch College \& The Graduate Center, CUNY.  New York, NY.}
\email{psoberon@gc.cuny.edu}

\thanks{This research was supported by NSF grant DMS 2051026.  Sober\'on's research is also supported by NSF grant DMS 2054419.}

\date{\today}

\begin{document}
\maketitle

\begin{abstract}
Krasnosselsky's art gallery theorem gives a combinatorial characterization of star-shaped sets in Euclidean spaces, similar to Helly's characterization of finite families of convex sets with non-empty intersection.  We study colorful and quantitative variations of Krasnosselsky's result.  In particular, we are interested in conditions on a set $K$ that guarantee there exists a measurably large set $K'$ such that every point in $K'$ can see every point in $K$.  We prove results guaranteeing the existence of $K'$ with large volume or large diameter.
\end{abstract}

\section{Introduction}

The study of intersection patterns of finite families of convex sets in Euclidean spaces is a central part of combinatorial geometry.  Helly gave a characterization for those families with non-empty intersection \cites{Radon:1921vh, Helly:1923wr} by proving that \textit{a finite family of convex set in $\rr^d$ has non-empty intersection if every subfamily of at most $d+1$ sets has non-empty intersection}.  There is now a multitude of variations and extensions of Helly's result \cites{Holmsen:2017uf, Amenta:2017ed}.

In 1946, Krasnosselsky proved a similar characterization of star-shapedness in $\rr^d$ \cite{Krasnosselsky1946}.  Given a set $K \subset \rr^d$, we say that two points $x,y$ in $K$ \textit{see} each other if the segment $[x,y]$ is contained in $K$.  We say that $K \subset \rr^d$ is \textit{star-shaped} if there exists $x$ in $K$ that sees all points in $K$.  Krasnosselsky proved the following.

\begin{theorem}[Krasnosselsky 1946]
	Let $K$ be a compact subset of $\rr^d$.  If for every $d+1$ or fewer points in $K$ there exists a point in $K$ that can see all of them, then there exists a point in $K$ that sees every point in $K$.
\end{theorem}

This result is sometimes called the art gallery problem, as we can interpret $K$ as the blueprint of a gallery, and the theorem gives conditions guaranteeing that a single guard can watch over the entire gallery.   The art gallery problem and its variations are relevant in computational geometry \cite{Bertschinger2022, ORourke1987}.

The proof of Krasnosselsky's theorem relies on a clever application of Helly's theorem.  The purpose of this manuscript is to show that several recent variations of Helly's theorem have a corresponding art gallery version.  This is known for some versions of Helly's theorem, such as fractional \cite{Kalai1997} and $(p,q)$ versions \cite{Barany:2006jh}.  Other known variations of Krasnosselsky's theorem include characterizations of families of convex sets with star-shaped union \cite{Breen:1990ez} or bounds on the dimension of the set of points that see the entire set (see Breen's survey and the references therein \cite{Breen1985}).

We focus on colorful and quantitative versions of Krasnosselsky's theorem.  The quantitative versions of Helly's theorem aim to characterize finite families of convex sets whose intersection is measurably large, rather than just non-empty.  The study of these families started with the volumetric versions of Helly by B\'ar\'any, Katchalski, and Pach \cite{Barany:1982ga}.  They proved the following.

\begin{theorem}[B\'ar\'any, Katchalski, Pach 1982]
	Given a finite family of convex sets in $\rr^d$, if the intersection of every $2d$ or fewer is non-empty and has volume at least one, then the intersection of the whole family has volume at least $d^{-2d^2}$.
\end{theorem}

Recently, quantitative Helly results gained notoriety when Nasz\'odi confirmed a conjecture by B\'ar\'any, Katchaslki, and Pach on the volume guarantees of the theorem above \cite{Naszodi:2016he}.  The use of similar analytic techniques has yielded similar results for other volumetric and diameter Helly-type theorems (see, e.g., \cites{Brazitikos:2016ja, FernandezVidal2022}).  Another successful approach to these problems involves higher-dimensional parametrizations of certain families of convex sets \cite{Sarkar2021, Dillon2021}.  We show how several of these results have guard gallery versions, in which we can guarantee that a large set of points can see an entire gallery—giving our single guard some room to move around.  For example, one of our results is the following volumetric version.

\begin{theorem}\label{thm:volume-version}
	Let $K \subset \rr^d$ be a compact set.  Suppose that for every $d(d+3)/2$ points in $K$ there exists a set $X\subset K$ of volume 1 such that every point of $X$ sees each of the $d(d+3)/2$ points.  Then, there exists a set $X' \subset K$ of volume $d^{-d}$ that sees every point of $K$.
\end{theorem}

In particular, we are interested in how the proof methods for quantitative Helly results involving higher-dimensional parametrizations relate to Krasnosselsky's theorem.  In \cref{sec:quantitative}, we describe general conditions on a parametrization of convex sets that imply quantitative Krasnosselsky-type results.  The conditions that lead to direct proofs of quantitative Krassnolesky-type results are the same as those needed to prove quantitative Tverberg-type theorems \cite{Sarkar2021}.  Understanding the nuances of the parametrization technique may have consequences beyond art gallery problems, which further motivates this approach.  We present results for volume and diameter.

One interesting aspect of Krasnosselsky-type results is that the usual examples showing optimality of Helly-type theorems do not carry directly to art galleries.  Therefore, one might wonder if the loss of volume is necessary in \cref{thm:volume-version}.  We present examples showing that this is needed, regardless of the number of points we might check.

\begin{theorem}\label{thm:non-exactness}
	Let $d,n$ be positive integers.  There exist $\varepsilon>0$ and a compact set $K \subset \rr^d$ such that for any $n$ points in $K$, there is a set $X \subset K$ of volume $1$ where every point in $X$ sees the $n$ points given, but there does not exist a set $X'$ of volume larger than $1-\varepsilon$ such that each point in $X'$ sees all of $K$.
\end{theorem}

The colorful versions of Helly's theorem are variations of Lov\'asz's colorful Helly theorem \cite{Barany1982}.  Lov\'asz proved that \textit{given $d+1$ finite families $\ff_1, \ldots, \ff_{d+1}$ of convex sets in $\rr^d$, if the intersection of every $(d+1)$-tuple $K_1 \in \ff_1, \ldots, K_{d+1} \in \ff_{d+1}$ is not empty, there exists an index $i \in \{1, \dots, d+1\}$ such that the intersection of $\ff_i$ is not empty.}  This result is called ``colorful'' because we can think of each $\ff_i$ as a color class.  There are a few ways to interpret what a colorful Krasnosselsky theorem would mean.  The following is an example of one interpretation we prove in the plane.

\begin{theorem}\label{thm-colorful-plane-krassno}
	Let $K$ be a compact, simply connected set in the plane and $P_1, P_2, P_3$ be finite subsets of $K$.  If for every choice $p_1 \in P_1, p_2 \in P_2, p_3 \in P_3$, we know there exists a point $x \in K$ that can see $p_1, p_2$, and $p_3$, there exist an index $i \in \{1,2,3\}$ and a point $x \in K$ that can see all of $P_i$.
\end{theorem}
While the proof method we use for this result cannot be extended to higher dimensions, we also prove a slightly different colorful version of Krasnosselsky's theorem that holds in any dimension.

In \cref{sec:colorful}, we present and prove our colorful Krasnosselsky results.  Then, in \cref{sec:quantitative}, we show the quantitative versions of Krasnosselsky's theorem as well as the examples proving non-exactness of certain Krasnosselsky-type results.  Finally, we include remarks and open problems in \cref{sec:remarks}.


\section{Colorful Krasnosselsky Results}\label{sec:colorful}

In this section, we show two different colorful versions of Krasnosselsky's theorem.  The first, \cref{thm-colorful-plane-krassno}, works in $\rr^2$ and only requires information about finite subsets of a simply connected region $K$.  The second, \cref{colorful krasnoselskii non-discrete}, works in $\rr^d$, but it involves stronger conditions.

We start with a short proof of \cref{thm-colorful-plane-krassno}.  For this proof, we will use a topological version of colorful Helly's theorem by Kalai and Meshulam \cite{Kalai:2005tb}.  We say that a good cover of a set $K \subset \rr^2$ is a family of sets whose union is $K$ and such that any finite intersection of them is either empty or contractible.  Applied to a simply connected region in $\rr^2$, Kalai and Meshulam's theorem states that \textit{given a good cover $\mathcal{G}$ of $K$ and three finite families $\mathcal{G}_1, \mathcal{G}_2, \mathcal{G}_3 \subset \mathcal{G}$ such that for any $F_1\in \mathcal{G}_1, F_2 \in \mathcal{G}_2, F_3 \in \mathcal{G}_3$ we know that $F_1 \cap F_2 \cap F_3$ has non-empty intersection, then at least one of the families $\mathcal{G}_i$ has non-empty intersection.}

\begin{proof}[Proof of Theorem \ref{thm-colorful-plane-krassno}]

Kalai and Matou\v{s}ek proved that if $K$ is a compact, simply connected region in the plane, and we denote by $V_x$ the visibility region of any $x \in K$, then every finite intersection of sets of the form $V_{x_1}, \ldots, V_{x_n}$ for $x_1, \dots, x_n \in K$ is either empty or contractible \cite{Kalai1997}.  In other words, the family of sets $\mathcal{G}= \{V_x : x \in K\}$ forms a good cover of $K$.
	
It suffices to consider $\mathcal{G}_i = \{V_x : x \in P_i\}$, and apply the Kalai--Meshulam theorem to finish the proof.
\end{proof}

First, the number of color classes in Theorem \ref{thm-colorful-plane-krassno} is optimal, as Figure \ref{fig:optimality} demonstrates.  Second, the topological conditions on $K$ are necessary.  If we remove the simply connected condition on $K$, Theorem \ref{thm-colorful-plane-krassno} fails in any dimension, even if we consider an arbitrary number of color classes. Figure \ref{fig:spider} gives an example in the plane.

\begin{claim}
	Let $n,d \ge 2$ be integers.  There exists a compact set $K\subset \rr^d$ and $n$ finite sets of points $\ff_1, \ldots \ff_n \subset K$ such that
	\begin{itemize}
		\item For any $n$ points $x_1 \in \ff_1, \ldots, x_n \in \ff_n$, there exists a point $p \in K$ that sees each of the $n$ points.
		\item For each $i \in [n]$, there is no point in $K$ that sees all of $\ff_i$.
	\end{itemize}
\end{claim}

\begin{proof}
	We start by taking arbitrary finite families of points $\ff_1, \ff_2, \ldots, \ff_n$ such that $\bigcup_{i=1}^n \ff_i$ is in general position and $|\ff_i| \ge 3$ for each $i$.  Then, for each $n$-tuple $P = (x_1, \ldots, x_n) \in \ff_1 \times \dots \times \ff_n$, we consider a point $g_P$ so that the set
	\[
	\{g_P : P \in \ff_1 \times \dots \times \ff_n\} \cup \ff_1 \cup \ldots \cup \ff_n
	\]
	is in general position.  For $P = (x_1, \dots, x_n)\in \ff_1 \times \dots \times \ff_n$, we include the segment $[x_i, g_P]$ in $K$ for each $i$.  If $d=2$, we also impose the condition that no three segments generated by pairwise disjoint pairs of the set above have a point in common.  When $d \ge 3$, the relative interior of each of these segments is disjoint from all the other segments generated this way.  When $d =2$, we have the additional condition.  Each point of $K$ can see at most one point of each $\ff_i$ if $d \ge 3$ and at most two points of each $\ff_i$ if $d=2$.
\end{proof}





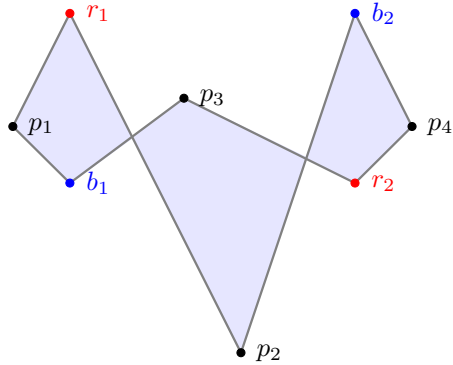
\begin{figure}[h]
\begin{center}
\begin{tikzpicture}[scale=0.75]
\filldraw[blue!10] (6,4) -- (7, 6) -- (8.1,3.8) -- (7, 3) -- cycle;
\filldraw[blue!10] (8.1,3.8) -- (9,4.5) -- (11.15,3.45) -- (10,0) -- cycle;
\filldraw[blue!10] (12,6) -- (13, 4) -- (12, 3) -- (11.15,3.45) -- cycle;

\draw[line width=0.3mm,gray](6,4)--(7,6);
\draw[line width=0.3mm,gray](6,4)--(7,3);
\draw[line width=0.3mm,gray](7,6)--(10,0);
\draw[line width=0.3mm,gray](7,3)--(9,4.5);
\draw[line width=0.3mm,gray](12,6)--(10,0);
\draw[line width=0.3mm,gray](13,4)--(12,6);
\draw[line width=0.3mm,gray](13,4)--(12,3);
\draw[line width=0.3mm,gray](12,3)--(9,4.5);

\filldraw [black] (6,4) circle (2pt);
\filldraw [red] (7,6) circle (2pt);
\filldraw [blue] (7,3) circle (2pt);
\filldraw [black] (9,4.5) circle (2pt);
\filldraw [black] (10,0) circle (2pt);
\filldraw [blue] (12,6) circle (2pt);
\filldraw [red] (12,3) circle (2pt);
\filldraw [black] (13,4) circle (2pt);

\node[color=red!100] at (7.5,6) {$r_1$};
\node[color=red!100] at (12.5,3) {$r_2$};
\node[color=blue!100] at (7.5,3) {$b_1$};
\node[color=blue!100] at (12.5,6) {$b_2$};
\node[color=black!100] at (6.5,4) {$p_1$};
\node[color=black!100] at (10.5,0) {$p_2$};
\node[color=black!100] at (9.5,4.5) {$p_3$};
\node[color=black!100] at (13.5,4) {$p_4$};
\end{tikzpicture}
\end{center}
\caption{A planar, simply-connected gallery in which each colorful pair can be seen but neither color class can be seen from a single point.}
\label{fig:optimality}
\end{figure}

\begin{figure}[h]
\begin{center}
    \begin{tikzpicture}
    \draw[line width=0.2mm,gray](5,3.1)--(4,3.5);
    \draw[line width=0.2mm,gray](5,3.1)--(6,3);
    \draw[line width=0.2mm,gray](5,3.1)--(6,4);

    \draw[line width=0.2mm,gray](7,2)--(6,3);
    \draw[line width=0.2mm,gray](7,2)--(6,4);
    \draw[line width=0.2mm,gray](7,2)--(7,1);

    \draw[line width=0.2mm,gray](1,5)--(4,3.5);
    \draw[line width=0.2mm,gray](1,5)--(6,4);
    \draw[line width=0.2mm,gray](1,5)--(2,0);

    \draw[line width=0.2mm,gray](9,0)--(7,1);
    \draw[line width=0.2mm,gray](9,0)--(6,4);
    \draw[line width=0.2mm,gray](9,0)--(2,0);

    \draw[line width=0.2mm,gray](4.5,2.8)--(4,3.5);
    \draw[line width=0.2mm,gray](4.5,2.8)--(6,3);
    \draw[line width=0.2mm,gray](4.5,2.8)--(3,2);

    \draw[line width=0.2mm,gray](5.5,2)--(6,3);
    \draw[line width=0.2mm,gray](5.5,2)--(3,2);
    \draw[line width=0.2mm,gray](5.5,2)--(7,1);

    \draw[line width=0.2mm,gray](2,3)--(4,3.5);
    \draw[line width=0.2mm,gray](2,3)--(3,2);
    \draw[line width=0.2mm,gray](2,3)--(2,0);

    \draw[line width=0.2mm,gray](5,1)--(7,1);
    \draw[line width=0.2mm,gray](5,1)--(3,2);
    \draw[line width=0.2mm,gray](5,1)--(2,0);
    
    \filldraw [black] (5,3.1) circle (2pt);
    \node[color=black!100] at (5.5,3.1) {$p_1$};
    \filldraw [black] (7,2) circle (2pt);
    \node[color=black!100] at (7.5,2) {$p_2$};
    \filldraw [black] (1,5) circle (2pt);
    \node[color=black!100] at (1.5,5) {$p_3$};
    \filldraw [black] (9,0) circle (2pt);
    \node[color=black!100] at (9.5,0) {$p_4$};
    \filldraw [black] (4.5,2.8) circle (2pt);
    \node[color=black!100] at (5,2.8) {$p_5$};
    \filldraw [black] (5.5,2) circle (2pt);
    \node[color=black!100] at (6,2) {$p_6$};
    \filldraw [black] (2,3) circle (2pt);
    \node[color=black!100] at (2.5,3) {$p_7$};
    \filldraw [black] (5,1) circle (2pt);
    \node[color=black!100] at (5.5,1) {$p_8$};

    \filldraw [red] (6,4) circle (2pt);
    \filldraw [red] (3,2) circle (2pt);
    \node[color=red!100] at (6.5,4) {$r_1$};
    \node[color=red!100] at (3.5,2) {$r_2$};
    \filldraw [green] (4,3.5) circle (2pt);
    \filldraw [green] (7,1) circle (2pt);
    \node[color=green!100] at (4.5,3.5) {$g_1$};
    \node[color=green!100] at (7.5,1) {$g_2$};
    \filldraw [blue] (2,0) circle (2pt);
    \filldraw [blue] (6, 3) circle (2pt);
    \node[color=blue!100] at (6.5,3) {$b_1$};
    \node[color=blue!100] at (2.5,0) {$b_2$};
    \end{tikzpicture}
\end{center}
\caption{A planar, non-simply connected gallery in which each colorful triple can be seen but no color class can be seen.}
\label{fig:spider}
\end{figure}
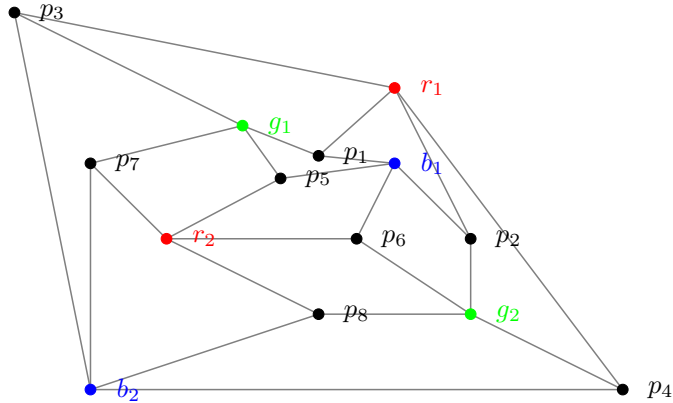

Figure \ref{fig:higher-dim} shows that we cannot apply the same approach of leveraging the topological version of colorful Helly's theorem to prove this colorful variant of Krasnosselsky's theorem in $\rr^d$.  However, it is still possible that this colorful combinatorial property holds in higher dimensional galleries that are simply connected or with other topological properties.

We now show the second colorful version of Krasnosselsky's theorem that holds in any dimension.  While this result does not have the same interpretation in the context of art galleries, it gives an interesting characterization of star-shaped sets. 

\begin{theorem} \label{colorful krasnoselskii non-discrete}
    Let $K \subset \rr^d$ be a compact set and $F_1, F_2, \ldots, F_{d+1}$ be closed subsets of $K$.  If for every choice $f_1 \in F_1, f_2, \in F_2, \ldots, f_{d+1} \in F_{d+1}$, we know there exists a point $x \in K$ such that $[x, f_i] \subset F_i$ for each $i \in [d+1]$, there exists some $i \in [d+1]$ such that $F_i$ is star-shaped.
\end{theorem}

The proof method uses a modification of Krasnosselsky's original arguments (see also \cite{Danzer:1963ug}).  The same technique gives us colorful versions of the quantitative theorems in the next section.

\begin{proof}

For each $x \in F_i$, where $i \in [d+1]$, let 
\begin{align*}
    V^{(i)}_x &= \{y \in \rr^d | [x, y] \subset F_i\}.
\end{align*}

Consider $\mathcal{G}_i = \{\conv(V^{(i)}_x) : x \in F_i\}$, where $1 \leq i \leq d+1$.  These are families of compact, convex sets in $\rr^d$.  If we choose one set from each $\mathcal{G}_i$, the resulting $(d+1)$-tuple will have non-empty intersection.  By colorful Helly's theorem for infinite families, we have that $\bigcap_{x \in F_i} \conv(V^{(i)}_x) \neq \emptyset$ for some $i$.  We will show that $\bigcap_{x \in F_i} \conv(V^{(i)}_x) \subset  \bigcap_{x \in F_i} V^{(i)}_x$, which will finish the proof.

Assume there exists some $y \in \bigcap_{x \in F_i} \conv(V^{(i)}_x)$ such that $y \notin \bigcap_{x \in F_i} V^{(i)}_x$.  Then, $y \notin V^{(i)}_x$ for some $x \in F_i$.  It follows that there exists $u \in [y, x)$ such that $u \notin F_i$.  There also exists $x' \in F_i$ such that $x' \in F_i \cap [u, x]$ and $F_i \cap [u, x') = \emptyset$.  Additionally, there exist $w \in (u, x')$ such that $\|w - x'\| = \dfrac{1}{2} \dist(\{u\}, F_i)$ and $v \in [u, w]$ and $x_0 \in F_i$ such that $\|x_0 - v\| = \dist([u, w], F_i)$.

Since $x_0$ is the point of $F_i$ that is closest to $v$, $V^{(i)}_{x_0}$ lies in the closed half-space $Q$ that is bounded by the hyperplane through $x_0$ and perpendicular to $[v, x_0]$.  Recall that since $y \in \bigcap_{x \in F_i} \conv(V^{(i)}_x)$, $y \in \conv(V^{(i)}_{x_0})$.  Since $V^{(i)}_{x_0} \subset Q$, $\conv(V^{(i)}_{x_0}) \subset Q$, so $y \in Q$. 
 It follows that $\langle v - x_0, y - x_0 \rangle \leq 0$, so $\angle yx_0v \geq {\pi}/{2}$.  Then, $\angle x_0vy < {\pi}/{2}$.  Note that $v \neq u$ since
    \begin{align*}
    \dist(\{v\}, F_i) &\leq \dist (\{w\}, F_i) \text{, by definition of $v \in [u, w]$} \\
                 &< \dist(\{u\}, F_i) \text{, since $\dist (\{w\}, F_i) \leq \|w-x'\| < \dist(\{u\}, F_i)$}.
    \end{align*}

So, there must exist a point in $[u, v)$ that is closer to $x_0$ than $v$ is.  This contradicts the selection of $v \in [u, w]$, completing the proof.
\end{proof}

\begin{figure}[h]
\begin{center}
\begin{tikzpicture}[scale=0.75]
\draw[dashed] (0,0,0) -- (0,5,0);
\draw (0,5,0) -- (5,5,0);
\draw (5,5,0) -- (5,0,0);
\draw[dashed] (0,0,0) -- (5,0,0);

\draw (0,0,5) -- (0,5,5) -- (5,5,5) -- (5,0,5) -- cycle;
	
\draw[dashed] (0,0,0) -- (0,0,5);
\draw (0,5,0) -- (0,5,5);
\draw (5,0,0) -- (5,0,5);
\draw (5,5,0) -- (5,5,5);

\draw[dashed] (1.25,5,2.5) arc (170:10:1.25cm and 0.4cm);
\draw (1.25,5,2.5) arc (-170:-10:1.25cm and 0.4cm);

\draw[dashed] (1.25,0,2.5) arc (170:10:1.25cm and 0.4cm);
\draw (1.25,0,2.5) arc (-170:-10:1.25cm and 0.4cm);

\draw (1.25, 5, 2.5) -- (2.5, 8, 2.5);
\draw (3.72, 5, 2.5) -- (2.5, 8, 2.5);
\draw (1.25, 0, 2.5) -- (2.5, 7, 2.5);
\draw (3.72, 0, 2.5) -- (2.5, 7, 2.5);

\node at (0, 4, 2.5) {\textbullet};
\node at (0, 3.75, 2.5) {$p_1$};
\node at (5, 4, 2.5) {\textbullet};
\node at (5, 3.75, 2.5) {$p_2$};

\node at (2.5,7.5,5) {\textbf{A}};
\node at (2.5,4,5) {\textbf{B}};

\draw[fill=blue!10] (9,0,5) -- (9,5,5) -- (10.25,5,5) -- (11.5,8,5) -- (12.75,5,5) -- (14,5,5) -- (14,0,5) -- (12.75,0,5) -- (11.5,7,5) -- (10.25,0,5) -- cycle;

\node at (9,4,5) {\textbullet};
\node at (9,3.75,5) {$p_1$};
\node at (14,4,5) {\textbullet};
\node at (14,3.75,5) {$p_2$};

\draw[dashed] (9,4,5) -- (12.75,7,5);
\draw[dashed] (14,4,5) -- (10.25,7,5);
\end{tikzpicture}
    
\caption{Gallery $K$ is defined to be the contractible region lying outside of cone $B$ and inside the union of the cube and cone $A$ on the left.  Take $p_1$ and $p_2$ to be two points on $\partial K$.  No points in $V_{p_1} \cap V_{p_2}$ lie in the plane containing $p_1$, $p_2$ and the center of the cube's bottom face shown in the cross section of $K$ on the right. 
 So, this plane induces a separation of $V_{p_1} \cap V_{p_2}$.  It follows that $V_{p_1} \cap V_{p_2}$ is disconnected and thus not contractible.}\label{fig: colorful counter 1}
\label{fig:higher-dim}
\end{center}
\end{figure}
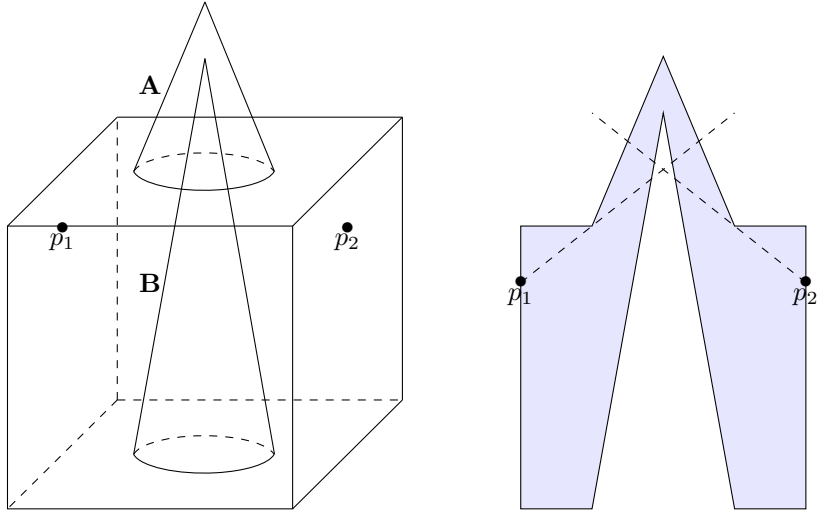

\section{Quantitative Krasnosselsky Results} \label{sec:quantitative}

Now we introduce the quantitative versions of Krasnosselsky's theorem. 
 Given a family of sets $\ff$ in $\rr^d$, we say that there is a Krasnosselsky theorem for $\ff$ if there exists a positive integer $n$ such that the following holds.\\

\emph{Let $K \subset \rr^d$ be compact.  Suppose that for any $n$ points in $K$, there exists a set $F \in \mathcal{F}$ so that $F \subset K$ and every point in $F$ sees each of the $n$ given points.  Then, there exists a set $F' \in \mathcal{F}$ such that $F' \subset K$ and every point in $F'$ sees all of $K$.}\\

If there is a Krasnosselsky theorem for $\ff$, we denote by $k_{\ff}$ the smallest value of $n$ for which the statement above holds.  We begin by showing general properties of $\ff$ that imply the existence of a Krasnosselsky theorem for $\ff$.  We then present results related to specific families that give rise to exact volumetric and diameter Helly-type theorems.  These results are closely related to recent quantitative Helly theorems \cites{Sarkar2021, Dillon2021}.



\begin{definition}
    Let $\ff$ be a family of sets in $\rr^d$.  We say that $\ff$ has a \textit{Krasnosselsky parametrization} in $\rr^l$ if there exists a convex subset $C \subset \rr^l$ and a function $D: C \to \ff$ such that
    \begin{itemize}
        \item The function $D$ is surjective.
        \item For all $a,b \in C$ and $\lambda \in [0,1]$, we have $D(\lambda a + (1-\lambda)b) \subset \lambda D(a) \oplus (1-\lambda)D(b)$, where $\oplus$ denotes the Minkowski sum in $\rr^d$.
    \end{itemize}
\end{definition}

The second condition on Krasnosselsky parametrizations has been used before to prove the existence of large families of quantitative Tverberg theorems \cite{Sarkar2021}.  We now show a general colorful quantitative Krasnosselsky theorem.

\begin{theorem}\label{thm-quant-krassno-general}
    Let $d,l$ be positive integers.  Let $\ff$ be a family of sets in $\rr^d$ that has a Krasnosselsky parametrization in $\rr^l$.  Let $K \subset \rr^d$ be compact and $F_1, \dots, F_{l+1}$ be closed subsets of $K$.  Suppose that for every choice $f_1 \in F_1, \dots, f_{l+1} \in F_{l+1}$ there exists a set $F \in \ff$ such that $\conv(\{f_i\} \cup F) \subset F_i$ for each $i \in [l+1]$.

    Then, there exist some $i \in [l+1]$ and a set $F \in \ff$ such that the segment between any point of $F_i$ and any point of $F$ is contained in $F_i$.
\end{theorem}

\begin{proof}
For each $x \in F_i$, let $V^{(i)}_x$ be the visibility set of $x$ in $F_i$.  In concrete terms, $V^{(i)}_x = \{y \in F_i : [x,y] \subset F_i\}$.  We know from the proof of Theorem \ref{colorful krasnoselskii non-discrete} that 
\[
\displaystyle \bigcap_{x \in F_i} \conv (V^{(i)}_x) \subset \bigcap_{x \in F_i}V^{(i)}_x.
\]
So, it is sufficient to prove that there exists an $i \in [l+1]$ and an $F \in \ff$ such that $F \subset \bigcap_{x \in F_i} \conv (V^{(i)}_x)$.

For each $i \in [l+1]$ and $x \in F_i$, let $C^i(x)$ be the set of elements in $\ff$ that $x$ can see completely in $F_i$.  Formally, 
\[
C^i(x) = \left\{y \in \rr^l: D(y) \subset V^{(i)}_x\right\}.
\]
For each $i \in [l+1]$, also let $\mathcal{G}_i = \{\conv(C^i(x)) : x \in F_i\}\subset \rr^l$.  These are families of convex sets in $\rr^l$.  The condition of the problem implies that if we pick an element from each $\mathcal{G}_i$, the resulting collection will have non-empty intersection.

Therefore, by the colorful Helly theorem in $\rr^l$, there exists some $i \in [l+1]$ such that $\bigcap \mathcal{G}_i$ has non-empty intersection. 
 Let $y$ be a point in $\bigcap \mathcal{G}_i$.  We will show that $D(y) = F \in \ff$ is the set we are looking for.

Take $x \in F_i$.  We want to show that $F \subset \conv (V^{(i)}_{x})$.  We know that $y \in \conv(C^i(x))$, so there must be some $y_1, \dots, y_n \in C^i(x)$ and coefficients of a convex combination $\alpha_1, \dots, \alpha_n$ such that $y = \alpha_1 y_1 + \dots + \alpha_n y_n$.  Note that $y_k \in C^i(x)$ means that $D(y_k) \subset V^{(i)}_x \subset \conv (V^{(i)}_x)$.  Then,
\begin{align*}
F  = D(y) & = D(\alpha_1 y_1 + \dots + \alpha_n y_n)
\\ & \subset \alpha_1 D(y_1) \oplus \dots \oplus \alpha_n D(y_n) \\
& \subset \alpha_1 \conv (V^{(i)}_x) \oplus \dots \oplus \alpha_n \conv (V^{(i)}_x)  \\
&= \conv (V^{(i)}_x).
\end{align*}
The last equality is due to the convexity of $\conv (V^{(i)}_x)$ and concludes the proof.
\end{proof}

\begin{corollary}
Let $\ff$ be a family of sets in $\rr^d$ that has a Krasnosselsky parametrization in $\rr^l$.  Then, $k_{\ff} \le l+1$.
\end{corollary}

\begin{proof}
    Let $K$ be a compact set in $\rr^d$.  It suffices to consider $F_1 = \dots = F_{l+1}=K$ to obtain the desired result.
\end{proof}

\subsubsection{\textbf{Axis Parallel Boxes}}

Let us showcase an application of \cref{thm-quant-krassno-general} with axis-parallel boxes.  We only state the non-colorful versions here, although \cref{thm-quant-krassno-general} implies colorful versions as well.

\begin{theorem}\label{boxes}
Let $d$ be a positive integer and $K \subset \rr^d$ be compact.  Suppose that for any $2d$ points in $K$, there exists an axis parallel box with volume $1$ so that every point in the box sees each of the given $2d$ points.  Then, there exists an axis-parallel box with volume $1$ so that every point in the box sees all of $K$.
\end{theorem}

\begin{proof}
We will show that the set of axis-parallel boxes of volume $1$ has a Krasnosselsky parametrization in $\rr^{2d-1}$.  Let $C$ be the set of points in $\rr^{2d-1}$ that have the last $d-1$ coordinates positive.  Given a point $p=(x_1,\dots, x_d, l_1, \dots, l_{d-1}) \in C$, we consider $l_d = 1/(l_1l_2\dots l_{d-1})$ and define $\boxx(p)$ as the axis-parallel box with corners at $(x_1, \dots, x_d)$ and $(x_1+l_1, \dots, x_d + l_d)$.

A direct application of the Brunn--Minkowski inequality shows that $\boxx(\lambda a + (1-\lambda)b) \subset \lambda \boxx(a) \oplus (1-\lambda)\boxx(b)$.
\end{proof}

An analogous approach shows the same result for 
sum of axis lengths.

\begin{theorem}\label{boxes2}
Let $d$ be a positive integer and $K \subset \rr^d$ be compact.  Suppose that for any $2d$ points in $K$, there exists an axis parallel box whose axis lengths sum to $1$ so that every point in the box sees each of the given $2d$ points.  Then, there exists an axis-parallel box whose axis lengths sum to $1$ so that every point in the box sees all of $K$.
\end{theorem}
\begin{proof}
In this case, we do not need to use the Brunn--Minkowski inequality, as the sum of axis lengths is linear in the space of axis-parallel boxes.  So, $\boxx(\lambda a + (1-\lambda)b) = \lambda \boxx(a) \oplus (1-\lambda)\boxx(b)$ with the natural parametrization. 
\end{proof}

\subsubsection{\textbf{Unit Spheres}}

\begin{theorem}
Let $K \subset \rr^d$ be compact.  Suppose that for any $d + 1$ points in $K$, there is unit sphere so that every point in the sphere sees all $d + 1$ points.  Then, there is a unit sphere so that every point in the sphere sees all of $K$.
\end{theorem}

\begin{proof}
We will show that the set of unit spheres has a Krasnosselsky parametrization in $\rr^{d}$. Let $C = \mathbb{R}^d$. Given $y_1, y_2 \in C$, consider $y^* = \lambda y_1 + (1-\lambda)y_2$ for some $\lambda \in [0, 1]$ and $v \in B(y^*, 1)$, where $v = y^* + \dfrac{\alpha}{||z||}z$, $z \in \mathbb{R}^d$, and $\alpha \in [0, 1]$.

Then, $v = y^* + \dfrac{\alpha}{||z||}z = \lambda y_1 + (1-\lambda)y_2 + \dfrac{\alpha}{||z||}z = \lambda(y_1 + \dfrac{\alpha}{||z||}z) + (1-\lambda)(y_2 + \dfrac{\alpha}{||z||}z) \in \lambda B(y_1, 1) \oplus (1-\lambda)B(y_2, 1)$.
\end{proof}

\subsubsection{\textbf{Ellipsoids}}

The case of ellipsoids is delicate.  Ellipsoids with axis lengths that sum to one have a Krasnosselsky parametrization, so we obtain the following theorem.

\begin{theorem}
Let $K \subset \rr^d$ be compact.  Suppose that for any $\dfrac{d(d+3)}{2}$ points in $K$, there is an ellipsoid with axis lengths that sum to one so that every point in the ellipsoid sees all $\dfrac{d(d+3)}{2}$ points. Then, there is an ellipsoid with axis lengths that sum to one so that every point in the ellipsoid sees all of $K$.
\end{theorem}

\begin{proof}
Since an ellipsoid is an affine image of a sphere, it follows that any ellipsoid can be represented by $a + XB_d$, where $B_d$ denotes the unit sphere in $\rr^d$, $X$ is non-singular, and $a \in \rr^d$.  We can find a polar decomposition $X = AQ$, where $A$ is a symmetric positive definite matrix and $Q$ is orthogonal.  Then, $XB_d = AQB_d = AB_d$.  Thus, any ellipsoid can be represented by $(a, A) \in \rr^{d(d+3)/2}$.  The sum of the lengths of the axes of $a +AB_d$ is equal to $2d\tr(A)$.

Let $P_d$ denote the space of symmetric, positive definite $d \times d$ matrices.  The set of positive definite matrices with trace equal to $1/2d$ is homeomorphic to $A \subset \rr^{d(d+1)/2 - 1}$ since $A$ is a co-dimension $1$ subspace of $P_d$.  The map $(a,A) \mapsto a + A B_d$ is a Krassnolesky parametrization of ellipsoids with unit axis length sum.


\end{proof}

Ellipsoids of unit volume do not have a direct Krasnosselsky parametrization, so we need a work-around.  We present an approach different from the one used by Sarkar et al. \cite{Sarkar2021}.

\begin{theorem}
Let $K \subset \rr^d$ be compact.  Suppose that for any $\dfrac{d(d+3)}{2}$ points in $K$, there is a unit volume ellipsoid so that every point in the ellipsoid sees all $\dfrac{d(d+3)}{2}$ points.  Then, there is a unit volume ellipsoid so that every point in the ellipsoid sees all of $K$.
\end{theorem}

\begin{proof}
    Denote by $B'_d$ the unit volume sphere in $\rr^d$ centered at the origin.  Let $Q_d \subset P_d$ be the set of symmetric, positive definite $d\times d$ matrices with determinant $1$.  Note that $\rr^d \times Q_d$ is homeomorphic to $\rr^{\frac{d(d+3)}{2}-1}$.  The map $(a, A)\to a+AB'_d$ parametrizes all ellipsoids of volume $1$ in $\rr^d$, but it is not a Krasnosselsky parametrization.

    To avoid this problem, we consider the set
     \[
   C(x) = \left \{ (a, A) \in \rr^d \times P_d \suchthat \det(A) \ge 1 \mbox{ and  }  a + AB_d' \subset \conv(V_x) \right \}.
    \]

    Since the determinant is log-concave in $P_d$, we get that $C(x)$ is a convex set.  Now, consider the map
    \begin{align*}
       \pi: \mathbb{R}^d \times P_d & \to \rr^d \times Q_d \\
       (a, A) & \mapsto \left(a, \frac{1}{\det (A)^{1/d}}A\right)
   \end{align*}

Note that $\pi (C(x)) \subset C(x)$ for all $x \in K$.  Suppose $M \subset \rr^d \times P_d$ is a convex set such that $\det(A) \ge 1$ for all $(a,A)\in M$.  If we restrict $\pi$ to $M$, the preimage of every point in the codomain of $\pi$ is a segment, which is contractible.  Therefore, $\pi (M)$ is homotopy equivalent to $M$ and also contractible.

  Consider the family
 \begin{align*}
     \mathcal{F} = \{\pi (C(x)) : x \in K \}.
 \end{align*}
Note that given a set $K' \subset K$, we have
\[
\bigcap_{x\in K'} \pi ( C(x))=\pi \left(\bigcap_{x\in K'}  C(x)\right) 
\]
This means that $\bigcap_{x\in K'} \pi ( C(x)) \subset \rr^d \times Q_d$ is either empty or contractible.

The conditions of the problem imply that every $d(d+3)/2$ or fewer sets in $\mathcal{F}$ have a non-empty intersection.  By the topological Helly theorem in dimension $(d(d+3)/2 )-1$ \cite{Helly:1930hk, Kalai:2005tb}, there must exist a point in the intersection of all of $\ff$. 

 This implies the existence of a volume one ellipsoid in $\bigcap_{x \in K} \conv (V_x)$, and we can conclude as before.

\end{proof}

To prove \cref{thm:volume-version}, we need a slight variation of the argument above.

\begin{proof}[Proof of \cref{thm:volume-version}]
    We follow the proof above, but instead of using $C(x)$, we use the set
    \[
   C'(x) = \left \{ (a, A) \subset \rr^d \times P_d \suchthat \det(A) \ge 1/d \mbox{ and  }  a + AB_d' \subset \conv(V_x) \right \}.
    \]
    Since $\vol (V_x) \ge 1$ for all $x$, we have $\vol (\conv (V_x)) \ge 1$.  Using John's ellipsoids (see, e.g., \cite{Ball:1997ud}), a convex set of volume $1$ contains an ellipsoid of volume $d^{-d}$.  This means that the family $\ff = \{\pi (\conv C'(x))\}$ satisfies that any $d(d+3)/2$ of its elements have a non-empty intersection, and we can conclude as before.
\end{proof}







\subsection{Diameter Results}\label{subsec:diameter}
Here, we present a few Krasnosselsky-type diameter results.  These results are concerned with formulating sufficient conditions to guarantee that $K \subset \rr^d$ can be seen by each point along some line segment.  We show that earlier methods to prove diameter Helly-type results \cite{Dillon2021} apply in the art gallery context.

\subsubsection{\textbf{V-width}}

\begin{definition}
Given a non-zero vector $v \in \rr^d$ and a compact set $M \subset \rr^d$, the $v$-width of $M$ is defined as $\max_{x, y \in M} \langle x - y, v \rangle$.
\end{definition}

\begin{theorem}
Let $v$ be a non-zero vector in $\rr^d$, and let $K \subset \rr^d$ be compact.  Suppose that for any $2d$ points in $K$, there is a line segment with $v$-width greater than or equal to 1 so that every point on the segment sees all $2d$ points.  Then, there is a line segment with $v$-width greater than or equal to 1 so that every point on the segment sees all of $K$.
\end{theorem}

\begin{proof}
For each $x \in K$, define
\begin{align*}
    C(x) &= \left \{(y, z) \in \rr^d \times \rr^d \suchthat [y, y + z] \subset V_x \text{ and } \langle z, v \rangle = 1 \right \}. 
\end{align*}

Note that $C(x) \subset \rr^{2d-1}$ due to the constraint $\langle z, v \rangle = 1$.  This is a Krasnosselsky parametrization of unit $v$-width segments, which concludes the proof.
\end{proof}

\subsubsection{\textbf{Minkowski norm}}

\begin{definition}
Given a compact convex set $K \subset \rr^d$ with nonempty interior that is symmetric about the origin, the Minkowski norm of $K$ is defined as 
\begin{align*}
    \rho_K(x) = \min\{t \geq 0 : x \in tK \}.
\end{align*}
\end{definition}

\begin{theorem}
\label{minkowski}
Let $\rho$ be a Minkowski norm in $\rr^d$ whose unit ball is a polytope $P$ with $k$ facets, and let $K \subset \rr^d$ be compact.  Suppose that for any $kd$ points in $K$, there is a line segment with $\rho$-norm greater than or equal to 1 so that every point on the segment sees all $kd$ points. 
 Then, there is a line segment with $\rho$-norm greater than or equal to 1 so that every point on the segment sees all of $K$.
\end{theorem}

\begin{proof}
For each $x \in K$, define
\begin{align*}
     C(x) &= \left \{(y, z) \in \rr^{2d} \suchthat [y, y + z] \subset V_x \text{ and }\rho(z) = 1 \right \}.
\end{align*}

Consider the family of convex sets $F = \{\conv(C(x))\}_{x \in K}$ in $\rr^{2d}$.  By the theorem hypothesis, we know that the intersection of any $kd$ members of this family contains a point in $\rr^d \times \partial P$. 
 We will make use of the following lemma to determine that $\bigcap F$ contains a point in $\rr^d \times \partial P$. 

\begin{lemma}[\cite{Dillon2021}]
\label{Dillon}
Let $P \subset \rr^d$ be a centrally symmetric polytope with $k$ facets and $\mathcal{G}$ be a finite family of convex sets in $\rr^{2d}$ such that 
\begin{itemize}
    \item $K \cap (\rr^d \times L)$ is convex for every facet $L$ of $P$ and every $K \in \mathcal{G}$.
    \item If $x, y \in \rr^d, K \in \mathcal{G}$, and $(x, y) \in K$, then $(x + y, -y) \in K$. 
\end{itemize}
If the intersection of every $kd$ or fewer sets in $\mathcal{G}$ contains a point in $\rr^d \times \partial P$, then $\cap \mathcal{G}$ contains a point in $\rr^d \times \partial P$.
\end{lemma}

Note that Lemma \ref{Dillon} requires that our family of convex sets be finite.  If we consider the family of compact, convex sets $F' = \{\conv(C(x)) \cap (\rr^d \times \partial P)\}_{x \in K}$, we know that any finite subfamily of $F'$ has nonempty intersection by this lemma.  So, by a basic property of compactness, the entire family $F'$ has nonempty intersection.

Choose $(y, z)$ in $\bigcap F$ and $\rr^d \times \partial P$.  We know that the segment $[y, y + z]$ has $\rho$-diameter greater than or equal to 1 since $\rho(z) = 1$.  For each $x \in K$, there exist coefficients $\alpha_1, \ldots, \alpha_n$ of a convex combination such that 
\begin{align*}
    (y, z) &= \sum_i \alpha_i(y_i, z_i) \text{, where $(y_i, z_i) \in C(x)$}.
\end{align*}
It follows that $[y, y + z] \subset \bigcap_{x \in K} \conv(V_x)$.  From which, $[y, y + z] \subset \bigcap_{x \in K} V_x$.
\end{proof}

\begin{theorem}
Let $K \subset \rr^d$ be compact.  Suppose that for any $2d^2$ points in $K$, there is a line segment with $\ell_p$-norm greater than or equal to 1 so that every point on the segment sees all $2d^2$ points.  Then, there is a line segment with $\ell_p$-norm greater than or equal to $d^{-1/p}$ so that every point on the segment sees all of $K$.
\end{theorem}

\begin{proof}
A set with $\ell_p$-diameter at least 1 has $\ell_{\infty}$-diameter at least $d^{-1/p}$.  Since the unit ball in the $\ell_{\infty}$ norm is a polytope with $2d$ facets, we can apply Theorem \ref{minkowski} to conclude that there exists a line segment that can see all of $K$ with $\ell_{\infty}$-diameter greater than or equal to $d^{-1/p}$.  It follows that the $\ell_{p}$-diameter of this segment is at least $d^{-1/p}$ as well.
\end{proof}

It was shown in \cite{Dillon2021} that only a Minkowski norm whose unit ball is a polytope admits an exact Helly-type theorem for diameter.  It would be interesting to explore whether the same is true in the Krasnosselsky context.\\

\subsection{Non-Exactness Examples}

In this section, we construct the examples for \cref{thm:non-exactness}. 
 Most examples showing that quantitative Helly-type theorems for volume or diameter are optimal or cannot be made exact rely on considering a family of half-spaces whose intersection is close to a unit sphere in Hausdorff distance.  We cannot make the visibility regions of points resemble such examples, so a more intricate construction is needed.

The goal is to construct a gallery $K \subset \rr^d$ which will look like a large ball (say, of radius $M'$).  Its boundary will have several spikes, so that the only set of points that can see the tip of each spike is close to a unit volume ball inside the gallery.  If done correctly, any small number $n$ of points in $K$ can be visible by a set of volume greater than $1+\delta$, where $\delta>0$ will only depend on $n$ and dimension $d$. Figure \ref{fig:non-exact} shows this construction in the plane.

\begin{proof}[Proof of \cref{thm:non-exactness}]Let $S^{d-1}$ denote the unit sphere centered at the origin in $\rr^d$, let $B'_d$ denote the ball of volume $1$ centered at the origin in $\rr^d$, and let $B_d$ denote the ball of radius $1$ centered at the origin in $\rr^d$.  Let $M'<M$ be two large real numbers, both larger than the radius of $B'_d$.  We consider $M'$ and $M$ to be very close to each other as well.  For each $v \in S^{d-1}$, let $C_v$ be the convex cone with apex $Mv$ over $B'_d$ (all its rays start at $Mv$ and go through a point of $B'_d$).  Let $n$ be a positive integer.  Now, consider the continuous function
\begin{align*}
	f: \left( S^{d-1}\right)^n & \to \rr^d \\
	(v_1, \dots, v_n) & \mapsto \vol \left( M'B_d \cap \left(\bigcap_{i=1}^n C_{v_i} \right)\right).
\end{align*}

For any $v \in S^{d-1}$, the boundaries of $C_v$ and $B'_d$ are only tangent on a set of dimension $(d-2)$.  Therefore, for any $v_1, \ldots, v_n \in S^{d-1}$, we have that $B'_d$ is a strict subset of $\bigcap_{i=1}^n C_{v_i}$.  This implies that $f(v_1,\ldots, v_n) >1$. 
 Since $f$ is a continuous function and $\left( S^{d-1}\right)^n$ is compact, it attains a minimum $m>1$.  

Let $\varepsilon = 1/2 - 1/(2m) > 0$.  We have $1/m = 1-2\varepsilon$.  Let $\delta >0$ be a sufficiently small real number such that $(1-2\varepsilon)(1+\delta) < (1-\varepsilon)$.  Let $S \subset S^{d-1}$ be a finite set such that
\[
\vol \left(\bigcap_{v \in S} C_v\right) = 1 + \delta \quad \mbox{and} \quad \bigcap_{v \in S} C_v \subset M' B_d.
\]
Now, we are ready to construct our gallery $K$.  For each $v \in S$, let $D_v = \conv (\{Mv\} \cup B'_d) \subset C_v$.  We define
\[
K = M'B_d \cup \left(\bigcup_{v \in S} D_v\right).
\]
This is a ball of radius $M'$ with many ``spikes'' around it.  For each $v \in S$, the points in $K$ that can see $Mv$ are precisely $C_v \cap K$. 
 Therefore, the set of points that can see all of $K$ is contained in  $\bigcap_{v \in S} C_v$ and must therefore have volume at most $1+\delta$.

Now, given any point $x \in K$, if $x \in C_v$, the visibility region $V_x$ satisfies $V_x \supset C_v \cap M' B_d$.  If $x \not\in C_v$ for any $v \in S$, then $V_x \supset M'B_d$.  Therefore, the volume of the visibility regions of any $n$ points in $K$ is bounded below by the volume of the visibility regions of $n$ points of $S$.  This, in turn, is bounded below by $m$.  Finally, we take the set $m^{-1/d}K$.  In this new gallery, for each $n$ points, there exists a set of volume at least $1$ that can see all of them.  Yet, the volume of the set of points that can see the entire set is at most $(1+\delta)/m = (1+\delta)(1-2\varepsilon) < (1-\varepsilon)$, as we wanted.

\begin{figure}[h]
\begin{center}
\scalebox{0.60}{
\begin{tikzpicture}

\filldraw[yellow!10] (0, 0) circle (3cm);

\filldraw[yellow!10] (3.732, 1) -- (2.828, 1) -- (2.95, 0.55) -- cycle;
\filldraw[yellow!10] (1, 3.732) -- (0.55, 2.95) -- (1, 2.828) -- cycle;
\filldraw[yellow!10] (-1, 3.732) -- (-0.55, 2.95) -- (-1, 2.828) -- cycle;
\filldraw[yellow!10] (-3.732, 1) -- (-2.828, 1) -- (-2.95, 0.55) -- cycle;
\filldraw[yellow!10] (-1, -3.732) -- (-1, -2.828) -- (-0.55, -2.95) -- cycle;
\filldraw[yellow!10] (1, -3.732) -- (1, -2.828) -- (0.55, -2.95) -- cycle;
\filldraw[yellow!10] (-3.732, -1) -- (-2.828, -1) -- (-2.95, -0.55) -- cycle;
\filldraw[yellow!10] (3.732, -1) -- (2.828, -1) -- (2.95, -0.55) -- cycle;

\filldraw[blue!10] (1, 1) -- (-0.577, 1) -- (-1.366, -0.366) -- (0, -1.1547) -- (1, -0.577) -- cycle;

\draw (0, 0) circle (1cm);
\draw (0, 0) circle (3cm);
\draw[line width=0.3mm, dotted] (0, 0) circle (3.864cm);

\filldraw [black] (0, 0) circle (1pt);
\filldraw [black] (0, 1) circle (1pt);
\filldraw [black] (0.5, -0.866) circle (1pt);
\filldraw [black] (-0.5, -0.866) circle (1pt);
\filldraw [black] (1, 0) circle (1pt);
\filldraw [black] (-0.866, 0.5) circle (1pt);

\draw[dotted] (0, 1) -- (3.732, 1);
\draw[dotted] (-1.5, -2.021) -- (3.732, 1);

\draw[dotted] (0, 1) -- (-3.732, 1);
\draw[dotted] (1.5, -2.021) -- (-3.732, 1);

\draw[dotted] (1, -2.3) -- (1, 3.732);
\draw[dotted] (-2, -1.464) -- (1, 3.732);

\filldraw [black] (3.732, 1) circle (1pt);
\filldraw [black] (-3.732, 1) circle (1pt);
\filldraw [black] (-1, 3.732) circle (1pt);
\filldraw [black] (1, 3.732) circle (1pt);
\filldraw [black] (-1, -3.732) circle (1pt);
\filldraw [black] (1, -3.732) circle (1pt);

\draw (-1, 3.732) -- (-0.55, 2.95);
\draw (-1, 3.732) -- (-1, 2.828);
\draw (-1, -3.732) -- (-1, -2.828);
\draw (-1, -3.732) -- (-0.55, -2.95);

\draw (1, 3.732) -- (0.55, 2.95);
\draw (1, 3.732) -- (1, 2.828);
\draw (1, -3.732) -- (1, -2.828);
\draw (1, -3.732) -- (0.55, -2.95);

\draw (3.732, 1) -- (2.828, 1);
\draw (3.732, 1) -- (2.95, 0.55);
\draw (-3.732, 1) -- (-2.828, 1);
\draw (-3.732, 1) -- (-2.95, 0.55);

\filldraw [black] (-3.732, -1) circle (1pt);
\draw (-3.732, -1) -- (-2.828, -1);
\draw (-3.732, -1) -- (-2.95, -0.55);

\filldraw [black] (3.732, -1) circle (1pt);
\draw (3.732, -1) -- (2.828, -1);
\draw (3.732, -1) -- (2.95, -0.55);

\filldraw [black] (-0.15, 3.4) circle (0.5pt);
\filldraw [black] (0, 3.4) circle (0.5pt);
\filldraw [black] (0.15, 3.4) circle (0.5pt);

\filldraw [black] (-0.15, -3.4) circle (0.5pt);
\filldraw [black] (0, -3.4) circle (0.5pt);
\filldraw [black] (0.15, -3.4) circle (0.5pt);

\filldraw [black] (-3.4, 0.15) circle (0.5pt);
\filldraw [black] (-3.4, 0) circle (0.5pt);
\filldraw [black] (-3.4, -0.15) circle (0.5pt);

\filldraw [black] (3.4, 0.15) circle (0.5pt);
\filldraw [black] (3.4, 0) circle (0.5pt);
\filldraw [black] (3.4, -0.15) circle (0.5pt);

\filldraw [black] (2.6, 2.35) circle (0.5pt);
\filldraw [black] (2.45, 2.5) circle (0.5pt);
\filldraw [black] (2.7, 2.2) circle (0.5pt);

\filldraw [black] (-2.6, 2.35) circle (0.5pt);
\filldraw [black] (-2.45, 2.5) circle (0.5pt);
\filldraw [black] (-2.7, 2.2) circle (0.5pt);

\filldraw [black] (2.6, -2.35) circle (0.5pt);
\filldraw [black] (2.45, -2.5) circle (0.5pt);
\filldraw [black] (2.7, -2.2) circle (0.5pt);

\filldraw [black] (-2.6, -2.35) circle (0.5pt);
\filldraw [black] (-2.45, -2.5) circle (0.5pt);
\filldraw [black] (-2.7, -2.2) circle (0.5pt);
\end{tikzpicture}}
\caption{Example of a gallery given by Theorem \ref{thm:non-exactness} in the plane.}
\label{fig:non-exact}
\end{center}
\end{figure}
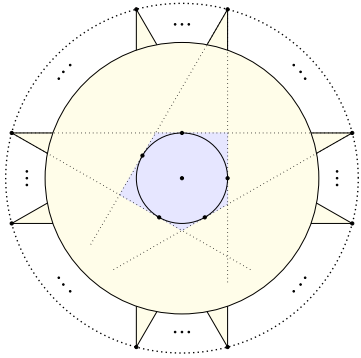
\end{proof}

\section{Future Directions and Remarks}\label{sec:remarks}

\subsection{Optimality of quantitative Krasnosselsky theorems.}

It is not clear whether the quantitative results we obtain in \cref{sec:quantitative} are optimal.  Specifically, determining the optimal dependence of $k_\mathcal{F}$ on the dimension $d$ for some family $\mathcal{F}$ in our quantitative Krasnosselsky theorems is of particular interest.  It is worth noting that the Helly number $2d$ in \cref{boxes}’s corresponding volumetric Helly theorem is optimal \cite{Xue2021}, as is the Helly number $kd$ in \cref{minkowski}’s corresponding diameter Helly theorem \cite{Dillon2021}.  It would be interesting to investigate whether such optimal quantitative Helly theorems correspond to art gallery problems that are optimal as well or that are better.  Identifying a quantitative Krasnosselsky theorem that is better than its Helly analogue would demonstrate an interesting difference between these two problems. 

\subsection{Colorful and fractional quantitative Krasnosselsky theorems.}

\cref{thm-quant-krassno-general} gives a general colorful quantitative Krasnosselsky theorem for families $\mathcal{F}$ that admit a Krasnosselsky parametrization.  Such results, however, are artificial in the context of the art gallery problem, so it is worth investigating whether there are other colorful quantitative versions of Krasnosselsky’s theorem that are more natural in the art gallery context.  Making \cref{thm-colorful-plane-krassno} quantitative and extending it to arbitrary dimension $d$, for example, would provide such a result.

\begin{problem}
Let $K \subset \mathbb{R}^d$ be compact and simply connected, and $P_1, \ldots, P_{2d}$ be finite subsets of $K$.  If for every choice $p_1 \in P_1, \ldots, p_{2d} \in P_{2d}$, we know there exists an axis parallel box of volume one so that every point in the box sees each of the $2d$ points, does there exist an index $i \in [2d]$ and an axis parallel box of volume one so that every point in the box sees all of $P_i$?
\end{problem}

Fractional versions of our quantitative results would provide insight into situations where only some $k_{\mathcal{F}}$-sized sets of points in a compact set $K$ can be seen by a set $F \in \mathcal{F}$, where $\mathcal{F}$ is some family of sets in $\mathbb{R}^d$.  A fractional variant of Krasnosselsky's theorem has been proven for planar galleries \cite{Barany:2006jh}.  We might consider attempting to make this fractional variant quantitative for different families of sets and for galleries of arbitrary dimension $d$.

\begin{problem}
For every $\alpha \in (0, 1)$ and positive integer $d$, does there exist a constant $\beta = \beta(\alpha, d)$ for which the following statement holds: Let $K \subset \mathbb{R}^d$ be compact and $A \subset K$ be an $n$-point set such that for at least $\alpha \binom{n}{2d}$ of the $2d$-tuples of points of $A$, there is an axis parallel box of volume one so that every point in the box sees each of the $2d$ points.  Then, there is an axis parallel box of volume one so that every point in the box sees at least $\beta n$ points of $A$.
\end{problem}

\subsection{Integer variants of Krasnosselsky’s theorem.}

It could be interesting to explore integer variants of Krasnosselsky’s theorem based on an integer version of Helly’s theorem proved by Jean-Paul Doignon in 1973 that states \textit{if $\mathcal{F}$ is a finite family of convex sets in $\mathbb{R}^d$ such that the intersection of every $2^d$ or fewer sets in $\mathcal{F}$ contains a point of $\mathbb{Z}^d$, then $\bigcap{F}$ contains a point of $\mathbb{Z}^d$} \cite{Doignon:1973ht}. 
 Interpreting this in the art gallery context leads to the following.

\begin{problem}
Let $K \subset \mathbb{R}^d$ be compact.  Suppose that for any $2^d$ points in $K$, there exists an integer point $x \in K$ that seems them all.  Then, does there exist an integer point $x \in K$ that sees all the points in $K$?
\end{problem}


\end{document}